\documentclass[11pt]{amsart}
\usepackage[latin1]{inputenc}
\usepackage{amssymb,epsfig}
\usepackage{amsmath, amsthm}
\usepackage{a4wide}
\usepackage{amsfonts}
\usepackage{float}
\usepackage{cite}
\usepackage{pdfsync}
\usepackage{color}
\usepackage{enumerate}

\newtheorem{theorem}{Theorem}

\newtheorem{corollary}{Corollary}

\newtheorem{remark}{Remark}

\title{A strong invariance principle for the elephant random walk}

\author{Cristian F. Coletti, Renato Gava and Gunter M. Sch\"utz}

\date{\today}

\address{
\newline 
UFABC - Centro de Matem\'atica, Computa\c{c}\~ao e Cogni\c{c}\~ao.
\newline
Avenida dos Estados, 5001, Santo Andr\'e - S\~ao Paulo, Brasil
\newline
e-mail:  \rm \texttt{cristian.coletti@ufabc.edu.br}
\newline 
\newline
UFSCAR - Departamento de Estat\'{\i}stica.
\newline  Rodovia Washington Luiz, Km 235, CEP 13565-905, S\~ao Carlos, Brasil
\newline
e-mail:  \rm \texttt{gava@ufscar.br} 
\newline
\newline
Institute of Complex Systems II
\newline
Forschungszentrum J\"ulich, 52425 J\"ulich, Germany
\newline
e-mail:  \rm \texttt{g.schuetz@fz-juelich.de}
}

\begin{document}

\begin{abstract}
We consider a non-Markovian discrete-time random walk on $\mathbb{Z}$ with unbounded memory called the elephant random walk (ERW). We prove a strong invariance principle for the ERW. More specifically, we prove that, under a suitable scaling and in the diffusive regime as well as at the critical value $p_c=3/4$ where the model is marginally superdiffusive, the ERW is almost surely well approximated by a Brownian motion. As a by-product of our result we get the law of iterated logarithm and the central limit theorem for the ERW. 
\end{abstract}

\maketitle

\section{Introduction and results} \label{intro}

In this work we deal with the so-called elephant random walk (ERW) introduced in \cite{ST}, 
a discrete-time nearest neighbour random walk $X_n$ on $\mathbb{Z}$ with long-range 
memory whose random increments at each time step depend on the entire history 
of the process. 
The ERW presents both normal and anomalous diffusion, i.e., depending on the parameter $p \in (0,1)$ 
the variance of $X_n$ either is a linear function of $n$ or not. More precisely, the 
ERW exhibits a phase transition at the critical value $p_c = 3/4$. If $p <  3/4$ the model lives in a diffusive 
regime, and if $p > 3/4$ it shows a superdiffusive regime.


The ERW has been received considerable attention in the literature in the last years, and results on it
range from exact moments \cite{ST,Para06,daSi13} to large deviation results \cite{Harr09}
up to connection with bond percolation on random recursive trees \cite{Kuer16} and P\'olya 
urn-types \cite{Baur16}. The extension of the ERW of \cite{Harb14} exhibits also subdiffusion.
The first limit theorems on the ERW show up in  \cite{Baur16,CGS}. In \cite{Baur16} functional 
limit theorems are derived for the diffusive regime as well as for the superdiffusive 
regime. The strong law of large numbers, the central limit theorem (CLT) and other limit theorems for 
the ERW are proven in \cite{CGS} for both regimes.  


We establish a strong invariance principle for the ERW that holds in the diffusive regime as 
well as at the critical value $p_c = 3/4$ where the model is marginally superdiffusive. This 
result guarantees that we can approximate the ERW almost surely by a standard Brownian motion. 
Furthermore, as a by-product of the strong invariance principle, we get a weak invariance principle 
for the ERW for $p \leq 3/4$ that shows an approximation in probability of the ERW by a Brownian 
motion. Both statements depend on different scales and give us explicit rates of approximation. 
Beyond its own interests, our findings prove the law of iterated logarithm (LIL) for the ERW, a 
limit theorem not covered in 
\cite{Baur16,CGS}, and extend the CLT of \cite{CGS}.

The way by which we study the long-time behaviour of the ERW depends deeply on martingale theory 
\cite{CR,HH,Wi}. The different manners in which the term invariance principle are employed can be 
observed in \cite{Ph,SS} and references therein. We notice that, similar to Donsker's theorem, the 
results of \cite{Baur16} are distribution invariance principles.







We now define the ERW as follows. It starts at $X_0 = 0$. At each discrete time step the 
elephant moves one step to the right or to the left respectively, so 
\begin{equation*}
X_{n+1} = X_n + \eta_{n+1},
\end{equation*}
where $\eta_{n+1} = \pm 1$ is a random variable. The memory consists of the set of random variables 
$\eta_{n^{\prime}}$ at previous time steps which the elephant remembers as follows.

The fist step goes to the right with probability $q$ and to the left with probability $1-q$.
At time $n+1$, for $n\geq 1$, a number $n^{\prime}$ from the set $\{1,2, \ldots , n\}$ is chosen at 
random with probability $1/n$. Then $\eta_{n+1}$ is determined stochastically by the rule 
\begin{align*}
\eta_{n+1} = \left\{\begin{array}{cl}
\eta_{n^{\prime}} & \text{ with probability } \, p \\
-\eta_{n^{\prime}} &  \text{ with probability } 1-p
\end{array}\right.
.
\end{align*}
We get from the definition that $\displaystyle X_n = \sum_{k=1}^n \eta_k$ and that
\begin{align}\label{conditional}
\mathbb{P}[\eta_{n+1} = \eta|\eta_1, \ldots, \eta_n] = \frac{1}{2n} \sum_{k=1}^n \left[1+\left(2p-1\right)\eta_k \eta\right] \ \mbox{for} \ n \geq 1,
\end{align}
where $\eta = \pm 1$. We recall that the anomalous diffusion is justified by the first and second moments of 
$X_n$ that we get from \cite{ST} 
\begin{align*}
\mathbb{E}[X_n] & = (2q-1) \frac{\Gamma(n+(2p-1))}{\Gamma(2p)\Gamma(n)} \sim \frac{2q-1}{\Gamma(2p)} n^{2p-1} \\
\mathbb{E}[X_n^{2}] & = \frac{n}{4p -3} \big[ \frac{\Gamma(n+4p-2)}{\Gamma(4p-2)\Gamma(n+1)} -1 \big]
 \sim \left\{\begin{array}{cl}
\frac{n}{3 - 4p} & \text{ for } \, p < 3/4 \\
n \log n & \text{ for } \, p = 3/4 \\
\frac{n^{4p -2}}{(4p -3)\Gamma(4p -2)} & \text{ for } \, p >3/4
\end{array}\right.
.
\end{align*}

Now we state our main result.

\begin{theorem}\label{strongip}
Let $(X_n)_{n \geq 1}$ be the ERW with $p \leq 3/4$ and let $\lbrace W_t \rbrace_{t \geq 0}$
be a standard Brownian motion. Then there exists a common probability space to $\lbrace X_n \rbrace_{n \geq 1}$ 
and $\lbrace W_t \rbrace_{t \geq 0}$ satisfying the following:
\begin{itemize}
\item[a)] If $p < 3/4$, then 
\begin{align}\label{iplil1}
 \left|\sqrt{3 - 4p}\dfrac{X_n }{n^{2p -1}}  - W\left( n^{3 - 4p} \right) \right| & =  \mbox{o}( n^{3/2 - 2p} \log \log n) 
 \quad \text{ a.s.}
\end{align}

\item[b)] If $p = 3/4$, then 
\begin{align} \label{iplil2}
 \left|\dfrac{X_n }{\sqrt{n}}  - W\left( \log n \right) \right| & =  \mbox{o}( \sqrt{ \log n \log \log \log n})
 \quad \text{ a.s.}
\end{align}

\item[c)] If $p < 3/4$, then 
\begin{align}\label{ipclt1}
n^{-(3/2 - 2p)} \left|\sqrt{3 - 4p} \dfrac{X_n }{n^{2p -1}}  - W\left( n^{3 - 4p} \right) \right| \xrightarrow{P} 0.
\end{align}

\item[d)] If $p = 3/4$, then 
\begin{align} \label{ipclt2}
(\log n)^{-1/2} \left|\dfrac{X_n }{\sqrt{n}}  - W\left( \log n \right) \right|  \xrightarrow{P} 0.
\end{align}

\end{itemize}
\end{theorem}

We announce below two straight consequences of Theorem \ref{strongip}.

\begin{corollary}[LIL]\label{lil}
Let $(X_n)_{n \geq 1}$ be the ERW and let $p \leq 3/4$. 
\begin{itemize}
\item[a)] If $p < 3/4$, then 
\begin{align}
\limsup_{n \to \infty}  \dfrac{|X_n|}{\sqrt{n \log \log n}} = \sqrt{\frac{2}{3 -4p}} \mbox{ a.s. } \nonumber
\end{align}

\item[b)] If $p = 3/4$, then
\begin{align}
\limsup_{n \to \infty}  \dfrac{|X_n|}{\sqrt{n \log n \log \log \log n}} = \sqrt{2} \mbox{ a.s. } \nonumber
\end{align}
\end{itemize}
\end{corollary}

\begin{corollary}[CLT]\label{clt}
Let $(X_n)_{n \geq 1}$ be the ERW and let $p \leq 3/4$. 
\begin{itemize}
\item[a)] If $p < 3/4$, then 
\begin{align}
\sqrt{3 - 4p}\dfrac{X_n }{\sqrt{n}} \xrightarrow{d} N(0,1). \nonumber
\end{align}

\item[b)] If $p = 3/4$, then
\begin{align}
 \dfrac{X_n}{\sqrt{n \log n}} \xrightarrow{d} N(0,1). \nonumber
\end{align}
\end{itemize}
\end{corollary}

\begin{remark}
If $p = 1/2$, the ERW becomes the simple symmetric random walk (SSRW) on $\mathbb{Z}$. See equation 
\eqref{conditional}. Note that our results on the invariance principles, LIL and CLT, are in accordance 
with the results in the literature for the SSRW.
\end{remark}

In the next section we present the proof of Theorem \ref{strongip}.

\section{Proof of Theorem \ref{strongip}}\label{proof}
Before presenting the proof of Theorem \ref{strongip} we collect some facts concerning the ERW in 
order to prove our theorem. 
Let
\begin{align*}
 a_1 := 1 \, \text{ and } \, a_{n} := \frac{\Gamma(n + 2p -1)}{\Gamma(n)\Gamma(2p)}  \, \text{ for } n \geq 2.
\end{align*}
Then define
\begin{align*}
M_n := \frac{X_n - \mathbb{E}[X_n]}{a_n}.
\end{align*}
In \cite{CGS} we showed that $\{ M_n \}_{n \geq 1} $ is a martingale with respect to the natural filtration $\mathcal{F}_n = \sigma(\eta_1, \ldots, \eta_n)$. Let
\begin{align*}
s^2_1 :=q(1-q) \, \text{ and } s^2_n := q(1-q) + \sum_{j=2}^n \frac{1}{a^2_j} \, \text{ for } n \geq 2.
\end{align*}
In \cite{CGS} we deduced that 
\begin{align}
a_{n} &\sim \frac{n^{2p -1}}{\Gamma(2p)}, \nonumber \\
s_n^{2} &\sim  \Gamma(2p)^{2} \frac{n^{3 - 4p}}{3 - 4p} \text{ if } p < 3/4,  \nonumber\\
s_n^{2} &\sim \Gamma(3/2)^{2} \log n, \text{ if } p =3/4, \label{asymptotic} \\
a_n s_n &\sim \sqrt{n/(3 - 4p)} \, \text{ if } p < 3/4,  \nonumber \\ 
a_n s_n &\sim \sqrt{n \log n} \, \text{ if } p =3/4. \nonumber 
\end{align}

\begin{proof}[Proof of Theorem \ref{strongip}]
Let us make a construction that will applied to the proof of \eqref{iplil1} and \eqref{iplil2} first and 
to \eqref{ipclt1} and \eqref{ipclt2} next. We begin by applying the Skorokhod embedding theorem for 
martingales (e.g.,  Theorem A.1, page 269 in \cite{HH}) to $M_n$. It states that there exists a new 
probability space where it is defined a standard Brownian motion $\lbrace W(t) \rbrace_{t \geq 0}$ 
and a sequence of non-negative r.v.'s $\tau_1, \tau_2, \ldots$ such that $W(T_n)$ has the same distribution 
of $M_n$ for any $n \geq 1$ where 
$T_n = \sum_{i = 1}^{n}\tau_i$ is $\mathcal{G}_n := \sigma(W(T_1), \ldots, W(T_n); W(t), 0 \leq t \leq T_n)$-measurable. 
If $D_1:=W(T_1)$ and $D_n:= W(T_n)-W(T_{n-1})$ for any $n \geq2$, then
\begin{align}\label{2r}
\mathbb{E}(\tau_j | \mathcal{G}_{j-1}) = \mathbb{E}(D_j^{2} | \mathcal{G}_{j-1}), \quad
\mathbb{E}(\tau_j^{r} | \mathcal{G}_{j-1}) \leq c_r \mathbb{E}(D_j^{2r} | \mathcal{G}_{j-1}) \, \mbox{ a.s. } 
\end{align}
for any $r \geq 1$ and for a positive constant $c_r$ depending only on $r$.

We know from \cite{CGS} that 
\begin{eqnarray}\label{ipdj}
\mathbb{E}\left[D_j^2 | \mathcal{G}_{j-1}\right] = \frac{1}{a^2_j}\left( 1 +  \mbox{o}(1) \right) \quad \text{a.s.}
\end{eqnarray}
Moreover, it comes out by (\ref{2r}) and (\ref{ipdj}) that
\begin{eqnarray}\label{tauG}
\sum_{j = 1}^{n}\mathbb{E}(\tau_j| \mathcal{G}_{j-1}) = s_n^2 (1 + \mbox{o}(1)) \quad \text{a.s.}
\end{eqnarray}

Now define $\hat{\tau}_j = \tau_j - \mathbb{E}(\tau_j| \mathcal{G}_{j-1})$, $j = 1, \ldots n$. Since $\hat{\tau}_j$
is a martingale difference with respect to $\mathcal{G}_j$, we obtain from \eqref{2r} and the fact that $|D_j| \leq 4/a_j$ (see \cite{CGS}) that
\begin{align*}
\mathbb{E}(\hat{\tau}_j^{2}|\mathcal{G}_{j-1})  = \mathbb{E}({\tau}_j^{2}| \mathcal{G}_{j-1}) - (\mathbb{E}({\tau}_j| \mathcal{G}_{j-1}))^{2} \leq c \frac{1}{a_j^{4}}
\end{align*}
for some positive constant $c$. It is easy to see from (\ref{asymptotic}), for any $p \leq 3/4$, that
\begin{align*}
\sum_{j =1}^{\infty} \mathbb{E}(\frac{\hat{\tau}_j^{2}}{s_j^{4}}|\mathcal{G}_{j-1}) & \leq c\sum_{j =1}^{\infty} \frac{1}{a_j^{4}s_j^{4}} < \infty.
\end{align*}
By Theorem 2.18 of \cite{HH} we have that $\displaystyle \sum_{j = 1}^{\infty}\frac{\hat{\tau}_j}{s_j^{2}} < \infty$ a.s. 
Kronecker's lemma (see section 12.7 of \cite{Wi}) implies that
\begin{align}\label{hattau}
\sum_{j = 1}^{n}\hat{\tau}_j = \mbox{o}(s_n^{2}) \quad \text{a.s.}
\end{align}
Combining (\ref{tauG}) and (\ref{hattau}) we get 
\begin{align}\label{tnsn}
T_n = \sum_{j = 1}^{n} \tau_j = s_n^{2} + \mbox{o}(s_n^{2}) \quad \text{a.s.}
\end{align}

Now we turn our attention to the proof of equations (\ref{iplil1}) and (\ref{iplil2}). By (\ref{tnsn}) and Theorem 1.2.1 in \cite{CR} we can assert that $W (T_n) = W(s_n^{2}) + \mbox{o}(\sqrt{s_n^{2}\log \log s_n})$ a.s. 

We recall from \cite{CGS} that
\begin{align}\label{mean}
\mathbb{E}[X_n] & \sim \frac{2q-1}{\Gamma(2p)} n^{2p-1}.
\end{align}
Therefore, it follows from \eqref{asymptotic} and \eqref{mean} that $\mathbb{E}[X_n]/a_n \sqrt{s_n^{2}\log \log s_n} \to 0$ for $p < 3/4$ as well as for $p = 3/4$. This finishes the proof of (\ref{iplil1}) and (\ref{iplil2}).

Now we use equation (\ref{tnsn}) to prove items \eqref{ipclt1} and \eqref{ipclt2} of Theorem \ref{strongip}. Given $\varepsilon > 0$ and $\delta > 0$, we have that
\begin{align*}
\mathbb{P}\left( \frac{|W(T_n) - W(s_n^{2})|}{s_n}  > \varepsilon \right) & =  
\mathbb{P}\left( \frac{|W(T_n) - W(s_n^{2})|}{s_n}  > \varepsilon , \, |T_n - s_n^{2}| > \delta s_n^{2} \right) \\
& + \mathbb{P}\left( \frac{|W(T_n) - W(s_n^{2})|}{s_n}  > \varepsilon , \, |T_n - s_n^{2}| \leq \delta s_n^{2} \right).
\end{align*}

Let us show that both terms of the last equality converge to zero when $n$ diverges to infinity. Since $|T_n - s_n^{2}|/ s_n^{2} \to 0$ a.s. as $n \to \infty$, we get 
\begin{align*}
\mathbb{P}\left( \frac{|W(T_n) - W(s_n^{2})|}{s_n}  > \varepsilon , \, |T_n - s_n^{2}| > \delta s_n^{2} \right) 
\leq \mathbb{P}\left( |T_n - s_n^{2}| > \delta s_n^{2} \right) \to 0 \, \mbox{ as } \, n \to \infty.
\end{align*}
On the other hand, after rescaling the Brownian motion, we get
\begin{align*}
\mathbb{P}\left( \frac{|W(T_n) - W(s_n^{2})|}{s_n}  > \varepsilon , \, |T_n - s_n^{2}| \leq  \delta s_n^{2} \right) 
\leq \mathbb{P}\left( \sup_{s; |s-1| \leq \delta} |W(s) - W(1)|  > \varepsilon \right)
\end{align*}
which goes to zero as $\delta \to 0$ since the trajectories of a Brownian motion are almost surely uniformly continuous on the compact set $[1 - \delta, 1]$. Now the proof of items \eqref{ipclt1} and  \eqref{ipclt2} of Theorem \ref{strongip} follows from the fact that $\mathbb{E}[X_n]/a_n s_n \to 0$ for $p < 3/4$ as well as for $p = 3/4$, by \eqref{asymptotic} and \eqref{mean}.
\end{proof}

\section{Acknoledegments}
The first author was partially supported by  by FAPESP (grant number
2015/20110-0 and 2016/11648-0). The second author was partially supported by CNPq (461365/2014-6).

\end{document}